\documentclass[a4paper,12pt,twoside]{article}	
\usepackage{amsmath,amssymb,amsthm,amsfonts,color,graphicx}
\usepackage{setspace}
\usepackage[margin=4cm]{geometry}
\newtheorem{theorem}{Theorem}
\DeclareMathOperator{\hr}{\mathbb H^2\times \mathbb R}
\DeclareMathOperator{\rr}{\mathbb R}
\DeclareMathOperator{\tor}{\mathbb T}
\DeclareMathOperator{\hh}{\mathbb H^2}
\newtheorem{lemma}{Lemma}

\newtheorem{remark}{Remark}

\theoremstyle{definition}\newtheorem{definition}{Definition}

\numberwithin{equation}{section}

\usepackage{layout,textcomp}
\title{On doubly periodic minimal surfaces in $\mathbb H^2 \times \mathbb R$ with finite total curvature in the quotient space}
\author{Laurent Hauswirth and Ana Menezes \footnote{ {\it{The authors were partially supported by the ANR-11-IS01-0002 grant. The second author was partially supported by CNPq-Brazil and IMPA.}}}}
\date{}

\begin{document}

\maketitle

\begin{abstract}
In this paper we develop the theory of properly immersed minimal surfaces in the quotient space $\hr/G,$ where $G$ is a subgroup of isometries generated by a vertical translation and a horizontal isometry in $\hh$ without fixed points. The horizontal isometry can be either a parabolic translation along horocycles in $\hh$ or a hyperbolic translation along a geodesic in $\hh.$ In fact, we prove that if a properly immersed minimal surface in $\hr/G$ has finite total curvature then its total curvature is a multiple of $2\pi,$ and moreover, we understand the geometry of the ends. These theorems hold true more generally for properly immersed minimal surfaces in $M\times\mathbb S^1,$ where $M$ is a hyperbolic surface with finite topology whose ends are isometric to one of the ends of the above spaces $\hr/G.$
\end{abstract}

\section{Introduction}
Among all the minimal surfaces in $\rr^3$, the ones of finite total curvature are the best known. In fact, if a minimal surface in $\rr^3$ has finite total curvature then this minimal surface is either a plane or its total curvature is a non-zero multiple of $2\pi.$ Moreover, if the total curvature is $-4\pi,$ then the minimal surface is either the Catenoid or the Enneper's surface \cite{O}.

In 2010, the first author jointly with Harold Rosenberg \cite{HR} developed the theory of complete embedded minimal surfaces of finite total curvature in $\hr.$ In that work they proved that the total curvature of such surfaces must be a multiple of $2\pi,$ and they gave simply connected examples whose total curvature is $-2\pi m,$ for each nonnegative integer $m.$

In the last few years, many people have worked on this subject and classified some minimal surfaces of finite total curvature in $\hr$ (see  \cite{HNST, HSET, MR, SET}). 

In \cite{MR} Morabito and Rodr\' iguez constructed for $k\geq 2$ a $(2k-2)$-parameter family of properly embedded minimal surfaces in $\hr$ invariant by a vertical translation which have total curvature $4\pi(1-k),$ genus zero and $2k$ vertical Scherk-type ends in the quotient by the vertical translation. Moreover, independently, Morabito and Rodr\' iguez \cite{MR} and Pyo \cite{P} constructed for $k\geq2$ examples of properly embedded minimal surfaces with total curvature $4\pi(1-k),$ genus zero and $k$ ends, each one asymptotic to a vertical plane. In particular, we have examples of minimal annuli with total curvature $-4\pi.$

It was expected that each end of a complete embedded minimal surface of finite total curvature in $\hr$ was asymptotic to either a vertical plane or a Scherk graph over an ideal polygonal domain. However in \cite{PR}, Pyo and Rodr\' iguez constructed new simply-connected examples of minimal surfaces of finite total curvature in $\hr,$ showing this is not the case.

In this work we consider $\hr$ quotiented by a subgroup of isometries $G\subset \mbox{Isom}(\hr)$ generated by a horizontal isometry in $\hh$ without fixed points, $\psi$, and a vertical translation, $T(h),$ for some $h>0.$ The isometry $\psi$ can be either a parabolic translation along horocycles in $\hh$ or a hyperbolic translation along a geodesic in $\hh.$ We prove that if a properly immersed minimal surface in $\hr/G$ has finite total curvature then its total curvature is a multiple of $2\pi,$ and moreover, we understand the geometry of the ends. More precisely, we prove that each end of a properly immersed minimal surface of finite total curvature in $\hr/G$ is asymptotic to either a horizontal slice, or a vertical geodesic plane or the quotient of a \textit{Helicoidal plane}. Where by \textit{Helicoidal plane} we mean a minimal surface in $\hr$ which is parametrized by $X(x,y)=(x,y,ax+b)$ when we consider the halfplane model for $\hh.$ 

Let us mention that these results hold true for properly immersed minimal surfaces in $M\times\mathbb S^1,$ where $M$ is a hyperbolic surface $(K_M=-1)$ with finite topology whose ends are either isometric to $\mathcal M_+$ or $\mathcal M_-,$ which we define in the next section.

\section{Preliminaries}
Unless otherwise stated, we use the Poincar\' e disk model for the hyperbolic plane, that is
$$
\mathbb H^2=\{(x,y)\in\rr^2|\ x^2+y^2<1\}
$$
with  the hyperbolic metric $g_{-1}=\sigma g_0=\frac{4}{(1-x^2-y^2)^2}g_0,$ where $g_0$ is the Euclidean metric in $\rr^2.$ In this model, the asymptotic boundary $\partial_{\infty}\mathbb H^2$ of $\mathbb H^2$ is identified with the unit circle and we denote by $p_o$ the point $(1,0)\in\partial_\infty\hh.$ 

We write $\overline{pq}$ to denote the geodesic arc between the two points $p,q.$

We consider the quotient spaces $\hr/G,$ where $G$ is a subgroup of Isom($\hr$) generated by a horizontal isometry on $\hh$ without fixed points, $\psi$, and a vertical translation, $T(h),$ for some $h>0.$ The horizontal isometry $\psi$ can be either a horizontal translation along horocycles in $\hh$ or a horizontal translation along a geodesic in $\hh.$

Let us analyse each one of these cases for $\psi.$

Consider any geodesic $\gamma$ that limits to $p_o$ at infinity parametrized by arc length. Let $c(s)$ be the horocycles in $\hh$ tangent to $p_o$ at infinity that intersects $\gamma$ at $\gamma(s)$ and write $d(s)$ to denote the horocylinder $c(s)\times\rr$ in $\hr.$ Taking two points $p,q \in c(s),$ let $\psi:\hr\rightarrow\hr$ be the parabolic translation along $d(s)$ such that $\psi(p)=q.$ We have $\psi(d(s))=d(s)$ for all $s.$ If $G=[\psi, T(h)],$ then the manifold $\mathcal M$ which is the quotient of $\hr$ by $G$ is diffeomorphic to $\tor^2\times \rr,$ where $\tor^2$ is the 2-torus. Moreover, $\mathcal M$ is foliated by the family of tori $\tor(s)=d(s)/G,$ which are intrinsically flat and have constant mean cuvature $\frac{1}{2}.$ (See Figure \ref{M2}).

\begin{figure}[h]
 \centering
\includegraphics[height=4.2cm]{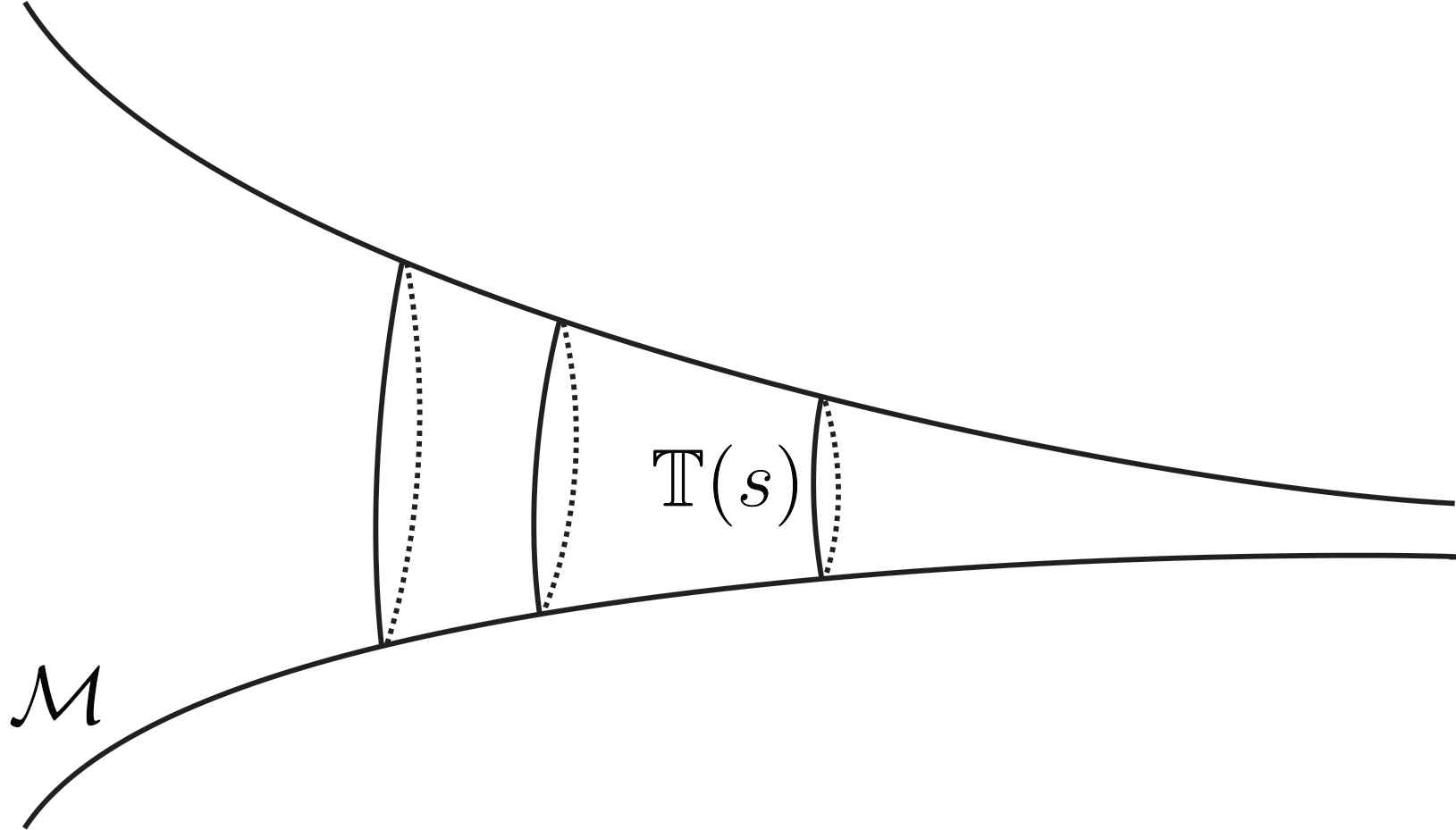}
\caption{$\mathcal M=\hr/[\psi, T(h)]$, where $\psi$ is a parabolic isometry.}
\label{M2}
\end{figure}

Now take a geodesic $\gamma$ in $\hh$ and consider $c(s)$ the family of equidistant curves to $\gamma,$ with $c(0)=\gamma.$ Write $d(s)$ to denote the plane $c(s)\times\rr$ in $\hr.$ Given two points $p,q\in c(s),$ let $\psi:\hr\rightarrow\hr$ be the hyperbolic translation along $\gamma$ such that $\psi(p)=q.$ We have $\psi(d(s))=d(s)$ for all $s.$ If $G=[\psi, T(h)],$ then the manifold $\mathcal M$ which is the quotient of $\hr$ by $G$ is also diffeomorphic to $\tor^2\times \rr$ and $\mathcal M$ is foliated by the family of tori $\tor(s)=d(s)/G,$ which are intrinsically flat and have constant mean cuvature $\frac{1}{2}$tanh$(s).$ (See Figure \ref{M11}).

\begin{figure}[h]
 \centering
\includegraphics[height=4.2cm]{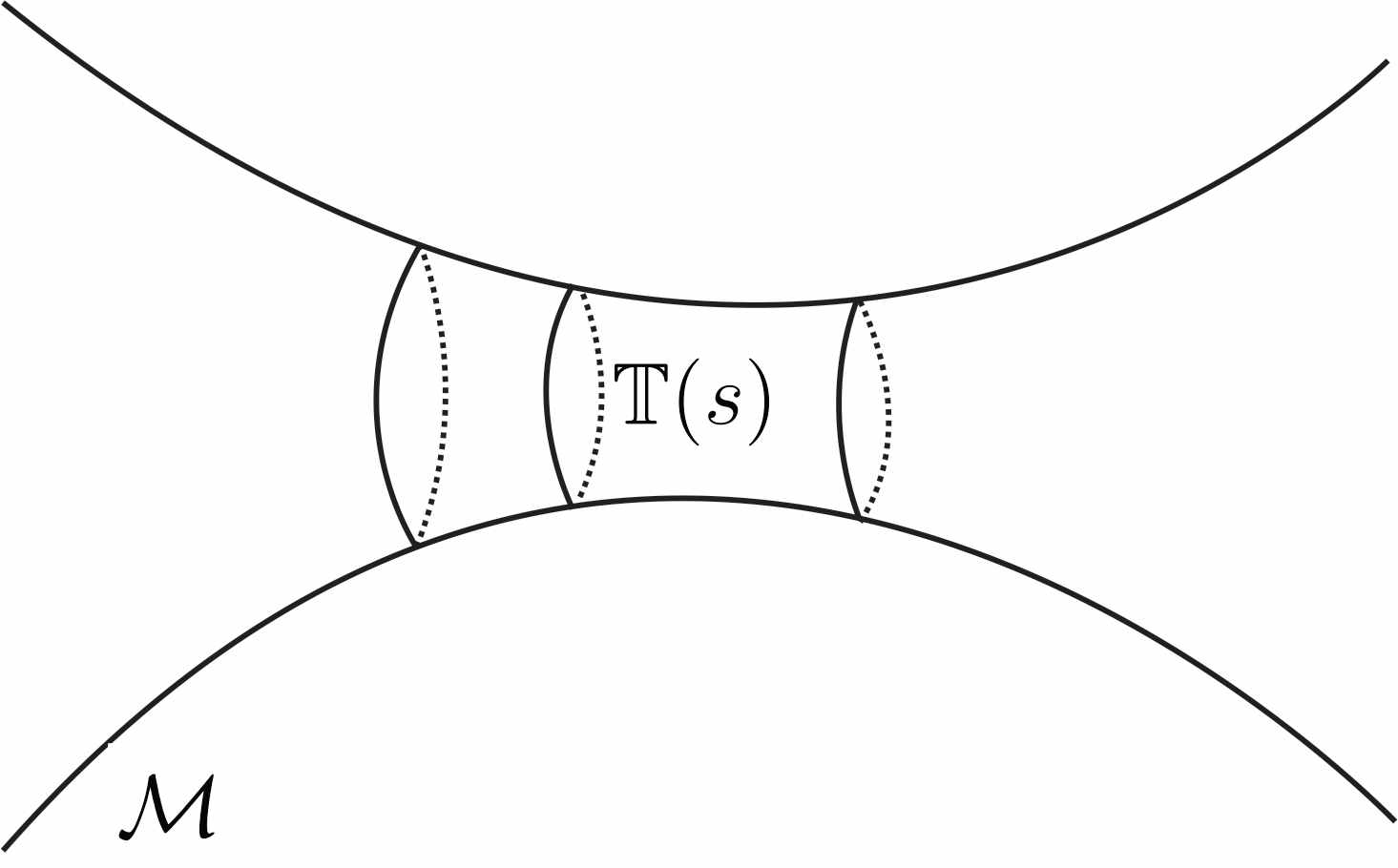}
\label{M11}
\caption{$\mathcal M=\hr/[\psi, T(h)]$, where $\psi$ is a hyperbolic isometry.}
\end{figure}

In these quotient spaces we have two different types of ends. One where the injectivity radius goes to zero at infinity, which we denote by $\mathcal M_+,$ and another one where the injectivity radius is strictly positive, which we denote by $\mathcal M_-$.

Hence $\mathcal M_+=\bigcup_{s\geq 0} d(s)/[\psi,T(h)],$ where $\psi$ is a parabolic translation along horocycles, and $\mathcal M_-=\bigcup_{s\geq 0} d(s)/[\psi,T(h)],$ for $\psi$ hyperbolic translation along a geodesic in $\hh,$ or $\mathcal M_-=\bigcup_{s\leq 0}d(s)/[\psi, T(h)],$ where $\psi$ can be either a parabolic translation along horocycles or a hyperbolic translation along a geodesic in $\hh.$ (See Figure \ref{M1eM21}).

\begin{figure}[h]
 \centering
\includegraphics[width=12.4cm]{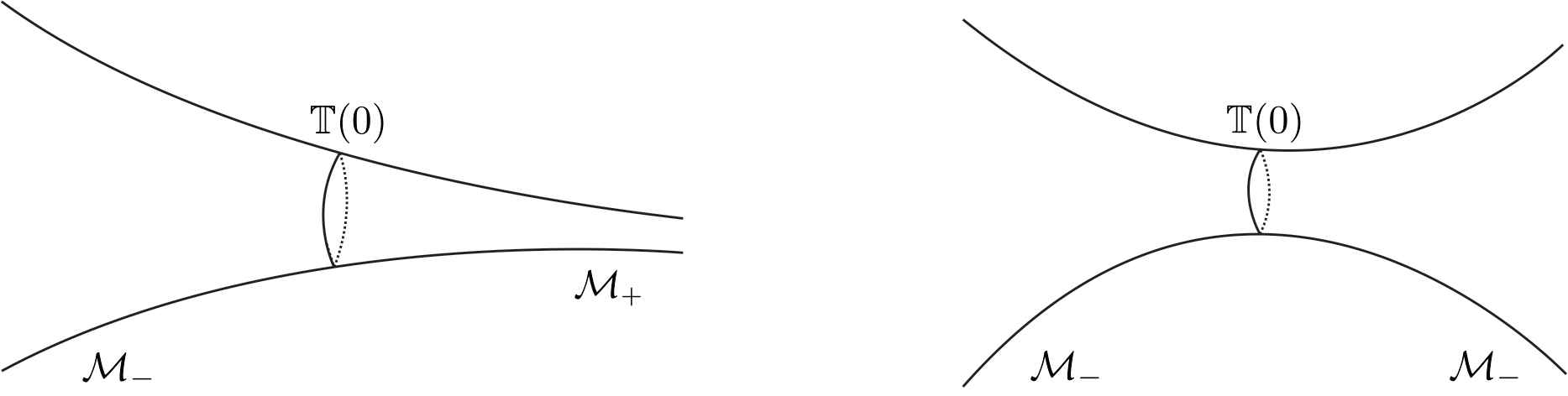}
\caption{$\mathcal M_+$ and $\mathcal M_-$.}
\label{M1eM21}
\end{figure}

From now one we will not distinguish between the two quotient spaces above. We will denote both by $\mathcal M.$ 

Let $\Sigma$ be a Riemannian surface and $X:\Sigma \rightarrow \mathcal M$ be a minimal immersion. As 
$$
\mathcal M=\hr/[\psi, T(h)]\cong\mathbb H^2/[\psi]\times \mathbb S^1,
$$
we can write $X=(F, h): \Sigma \rightarrow \mathbb H^2/[\psi]\times \mathbb S^1,$ where $F: \Sigma \rightarrow \mathbb H^2/[\psi]$ and $h: \Sigma \rightarrow \mathbb S^1$ are harmonic maps. We consider local conformal parameters $z= x+iy$ on $\Sigma.$ Hence
\begin{equation}
\begin{array}{ccc}
|F_x|^2_{\sigma}+(h_x)^2= |F_y|^2_{\sigma}+(h_y)^2&&\\
&&\\
\left\langle F_x,F_y\right\rangle_{\sigma}+h_x.h_y=0&&\\
\end{array}
\label{conf}
\end{equation}
and the metric induced by the immersion is given by 
\begin{equation}
ds^2=\lambda^2(z)|dz|^2=(|F_z|_\sigma+|F_{\bar{z}}|_\sigma)^2|dz|^2.
\label{metric}
\end{equation}

Considering the universal covering $\pi:\hr\rightarrow\mathbb H^2/[\psi]\times \mathbb S^1$ we can take $\widetilde \Sigma,$ a connected component of the lift of $\Sigma$ to $\hr,$ and we have $\widetilde{X}=(\widetilde{F}, \widetilde{h}): \widetilde{\Sigma}\rightarrow \hr$ such that $\pi(\widetilde{\Sigma})=\Sigma$ and $\widetilde{F}: \widetilde\Sigma \rightarrow \mathbb H^2, \widetilde{h}: \widetilde\Sigma \rightarrow \mathbb R$ are harmonic maps. We denote by $\widetilde{\partial_t}, \partial_t$ the vertical vector fields in $\hr$ and $\mathbb H^2/[\psi]\times \mathbb S^1,$ respectively. Observe that the functions $n_3: \Sigma \rightarrow \rr,$ $\tilde{n}_3: \widetilde\Sigma \rightarrow \rr,$ given by $n_3=\left\langle \partial_t, N \right\rangle, \tilde{n}_3=\left\langle \widetilde{\partial_t}, \widetilde{N} \right\rangle,$ where $N, \widetilde{N}$ are the unit normal vectors of $\Sigma, \widetilde\Sigma,$ respectively, satisfy $\tilde{n}_3=n_3\circ\pi.$ Then if we define the functions $\omega: \Sigma \rightarrow \rr, \tilde\omega: \widetilde\Sigma \rightarrow \rr$ so that $\mbox{tanh}(\omega)=n_3$ and $\mbox{tanh}(\tilde\omega)=\tilde{n}_3,$ we get  $\tilde\omega=\omega\circ\pi.$

As we consider $X$ a conformal minimal immersion, we have
\begin{equation}
n_3=\frac{|F_z|^2-|F_{\bar{z}}|^2}{|F_z|^2+|F_{\bar{z}}|^2}
\label{n3}
\end{equation}
and
\begin{equation}
\omega=\frac{1}{2}\ln\frac{|F_z|}{|F_{\bar{z}}|}.
\label{omegaln}
\end{equation}

Note that the same formulae are true for $\tilde{n}_3$ and $\tilde\omega.$

We know that for local conformal parameters $\tilde{z}$ on $\widetilde\Sigma,$ the holomorphic quadratic Hopf differential associated to $\widetilde{F}$, given by $$\widetilde{Q}(\widetilde{F})=(\sigma\circ \widetilde {F})^2\widetilde{F}_{\tilde{z}}\bar{\widetilde{F}}_{\tilde{z}}(d\tilde{z})^2,$$ can be written as $(\widetilde{h}_{\tilde{z}})^2(d\tilde{z})^2=-\widetilde{Q}$. Then, since $\widetilde{h}$ and $h$ differ by a constant in a neighborhood, $(h_z)^2(dz)^2=-Q$ is also a holomorphic quadratic differential on $\Sigma$ for local conformal parameters $z$ on $\Sigma.$ We note $Q$ has two square roots globally defined on $\Sigma.$ Writing $Q=\phi(dz)^2,$ we denote by $\eta=\pm2i\sqrt{\phi}dz$ a square root of $Q,$ where we choose the sign so that
$$
h=\mbox{Re}\ \int\eta.
$$

Using (\ref{metric}), (\ref{omegaln}) and the definition of $Q,$ we have
\begin{equation}
ds^2=4(\mbox{cosh}^2\omega)|Q|.
\label{ds}
\end{equation}

As the Jacobi operator of the minimal surface $\Sigma$ is given by
$$
J=\frac{1}{4\cosh^2\omega|\phi|}\left[\Delta_0-4|\phi|+\frac{2|\nabla\omega|^2}{\cosh^2\omega}\right]
$$
and $Jn_3=0,$ then
\begin{equation}
\Delta_0\omega=2\sinh(2\omega)|\phi|,
\label{deltaomega}
\end{equation}
where $\Delta_0$ denotes the Laplacian in the Euclidean metric $|dz|^2$, that is, $\Delta_0=4\partial^2_{z\bar{z}}.$ 

The sectional curvature of the tangent plane to $\Sigma$ at a point $z$ is $-n_3^2$ and the second fundamental form is
$$
II=\frac{\omega_x}{\mbox{cosh}\omega}dx\otimes dx-\frac{\omega_x}{\mbox{cosh}\omega}dy\otimes dy+2\frac{\omega_y}{\mbox{cosh}\omega}dx\otimes dy.
$$

Hence, using the Gauss equation, the Gauss curvature of $(\Sigma, ds^2)$ is given by
\begin{equation}
K_{\Sigma}=-\mbox{tanh}^2\omega-\frac{|\nabla \omega|^2}{4(\mbox{cosh}^4\omega)|\phi|}.
\label{eqKK}
\end{equation}

\section{Main results}
In this section, besides prove the main theorem of this paper, we will firstly demonstrate some properties of an end when it is properly immersed in $\mathcal M_+$ or in $\mathcal M_-,$ which are interesting by theirselves.

We will write $[d(0),d(s)]$ to denote the slab $\cup_{0\leq t \leq s}d(t)$ in $\hr$ whose boundary is $d(0)\cup d(s).$

\begin{lemma}
There is no proper minimal end $E$ in $\mathcal M_+$ with $\partial \mathcal M_+\cap E=\partial E$ whose lift is an annulus in $\hr.$
\label{annulus}
\end{lemma}
\begin{proof}

Let us prove it by contradiction. Suppose we have a proper minimal end $E$ in $\mathcal M_+$ with $\partial \mathcal M_+\cap E=\partial E$ whose lift $\widetilde E$ is a proper minimal annulus in $\hr.$ Hence $\partial \widetilde E\subset d(0)$, $\widetilde E\subset \bigcup_{s\geq0}d(s)$ and $\widetilde{E}\cap d(s)\neq\emptyset$ for any $s,$ where $d(s)=c(s)\times\rr,$ $c(s)$ horocycle tangent at infinity to $p_o.$ 

Choose $p\neq p_o\in\partial_\infty \hh$ such that $(\overline{pp_o}\times\rr)\cap\partial\widetilde {E}=\emptyset.$


Now consider $q\in\partial_\infty\hh$ contained in the halfspace determined by $\overline{pp_o}\times\rr$ that does not contain $\partial\widetilde E$ such that $(\overline{pq}\times\rr)\cap d(0)=\emptyset.$ Let $q$ go to $p_o.$ If there exists some point $q_1$ such that $(\overline{pq_1}\times\rr)\cap \widetilde E\neq\emptyset,$ then, as $p,q_1\notin d(s)$ for any $s$, and $E$ is proper, that intersection is a compact set in $\widetilde E.$ Therefore, when we start with $q$ close to $p$ and let $q$ go to $q_1,$ there will be a first contact point between $\overline{pq_0}\times\rr$ and $\widetilde E,$ for some point $q_0.$ By the maximum principle this yields a contradiction. Therefore, we conclude that $\overline{pp_0}\times\rr$ does not intersect $\widetilde E.$ Choosing another point $\bar{p}$ in the same halfspace determined by $\overline{pp_o}\times\rr$ as $\widetilde E$ such that $(\overline{\bar{p}p_o}\times\rr)\cap\partial\widetilde {E}=\emptyset,$ we can use the same argument above and conclude that $\widetilde{E}$ is contained in the region between $\overline{pp_o}\times\rr$ and $\overline{\bar{p}p_o}\times\rr.$ Call $\alpha=\overline{pp_o}$ and $\bar{\alpha}=\overline{\bar{p}p_o}.$

\begin{figure}[h]
 \centering
\includegraphics[height=4cm]{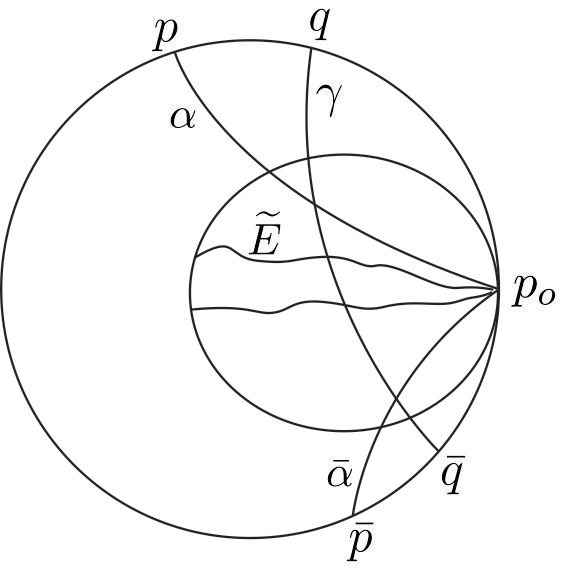}
\caption{Curve $\gamma$.}
\label{gamma1}
\end{figure}

Now consider a horizontal geodesic $\gamma$ with endpoints $q,\bar{q}$ such that $q$ is contained in the halfspace determined by $\alpha\times\rr$ that does not contain $\widetilde E,$ and $\bar{q}$ is contained in the halfspace determined by $\bar{\alpha}\times\rr$ that does not contain $\widetilde E$ (see Figure \ref{gamma1}). Up to translation, we can suppose $\widetilde E\cap(\gamma\times\rr)\neq\emptyset.$ As $E$ is proper, the part of $\widetilde E$ between $\partial\widetilde E$ and $\widetilde E\cap\left( \gamma\times\rr\right)$ is compact, then there exists $M\in\rr$ such that the function $\widetilde h$ restrict to this part satisfies $-M\leq \widetilde h\leq M.$ Consider the function $v$ that takes the value $+\infty$ on $\gamma$ and take the value $M$ on the asymptotic arc at infinity of $\hh$ between $q$ and $\bar{q}$ that does not contain $p_o$. The graph of $v$ is a minimal surface that does not intersect $\widetilde E.$ When we let $q,\bar{q}$ go to $p_o$ we get, using the maximum principle, $\widetilde E$ is under the graph of $v$ and then $\widetilde h|_{\widetilde E}$ is bounded above by $M$, since $v$ converges to the constant function $M$ uniformly on compact sets as $q,\bar{q}$ converge to $p_o$ (see section B, \cite{MMR}). Using a similar argument, we can show that $\widetilde h|_{\widetilde E}$ is also bounded below by $-M$. Therefore $\widetilde E$ is an annulus contained in the region bounded by $\alpha\times\rr, \bar{\alpha}\times\rr, \hh\times\{-M\}$ and $\hh\times\{M\}.$

Take four points $p_1,p_2,p_3,p_4\in\partial_{\infty}\hh$ such that $p_1,p_2$ is contained in the halfspace determined by $\alpha\times\rr$ that does not contain $\widetilde E,$ and $p_3,p_4$ is contained in the halfspace determined by $\bar{\alpha}\times\rr$ that does not contain $\widetilde E$. Moreover, choose these points so that there exists a complete minimal surface $\mathcal A$ taking value $0$ on $\overline{p_1p_2}$ and $\overline{p_3p_4},$ and taking value $+\infty$ on $\overline{p_2p_4}$ and $\overline{p_1p_3}$ (see Figure \ref{figA}). This minimal surface exists by \cite{CR}. 

\begin{figure}[h]
 \centering
\includegraphics[height=4cm]{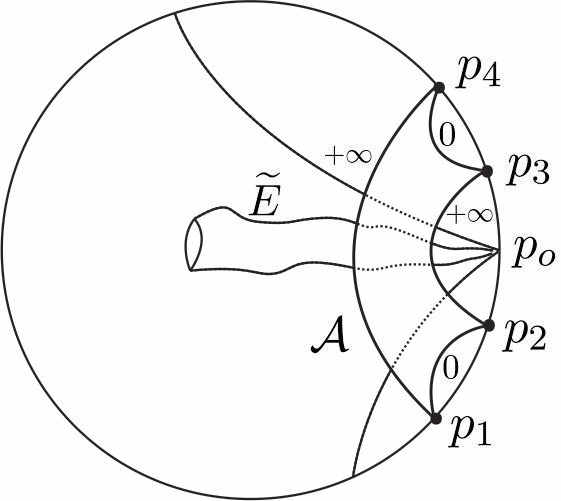}
\caption{Minimal graph $\mathcal A$.}
\label{figA}
\end{figure}

Up to a vertical translation, $\mathcal A$ does not intersect $\widetilde E$ and $\mathcal A$ is above $\widetilde E.$ Pushing down $\mathcal A$ (under vertical translation) and using the maximum principle, we conclude that $\mathcal A=\widetilde E,$ what is impossible. 
\end{proof}

\begin{remark}
We do not use any assumption on the total curvature of the end to prove the previous lemma.
\end{remark}

\begin{lemma}
\label{prop-area}
If a proper minimal end $E$ with finite total curvature is contained in $\mathcal M_-$, then $E$ has bounded curvature and infinite area.
\end{lemma}

\begin{proof}

Suppose $E$ does not have bounded curvature. Then there exists a divergent sequence $\{p_n\}$ in $E$ such that $|A(p_n)|\geq n,$ where $A$ denotes the second fundamental form of $E.$ As the injectivity radius of $\mathcal M_-$ is strictly positive, there exists $\delta>0$ such that for all $n,$ the exponential map exp$_{\mathcal M}: D(0,\delta)\subset T_{p_n}\mathcal M\rightarrow B_{\mathcal M}(p_n,\delta)$ is a diffeomorphism, where  $B_{\mathcal M}(p_n,\delta)$ is the extrinsic ball of radius $\delta$ centered at $p_n$ in $\mathcal M.$ Without loss of generality, we can suppose $B_{\mathcal M}(p_n,\delta)\cap B_{\mathcal M}(p_k,\delta)=\emptyset.$

The properness of the end implies the existence of a curve $c\subset E$ homotopic to $\partial E$ such that every point in the connected component of $E\setminus c$ that does not contain $\partial E$ is at a distance greater than $\delta$ from $\partial E.$ Call $E_1$ this component. Hence each point of $E_1$ is the center of an extrinsic ball of radius $\delta$ disjoint from $\partial E.$

Denote by $C_n$ the connected component of $p_n$ in $B_{\mathcal M}(p_n,\delta)\cap E_1$ and consider the function $f_n:C_n\rightarrow\rr$ given by 
$$
f_n(q)=d(q,\partial C_n)|A(q)|,
$$
where $d$ is the extrinsic distance.

The function $f_n$ restricted to the boundary is identically zero and $f_n(p_n)=\delta|A(p_n)|>0.$ Then $f_n$ attains a maximum in the interior. Let $q_n$ be such maximum. Hence $\delta|A(q_n)|\geq d(q_n,\partial C_n)|A(q_n)|=f_n(q_n)\geq f_n(p_n)=\delta|A(p_n)|\geq \delta n,$ what yields $|A(q_n)|\geq n.$

Now consider $r_n=\frac{d(q_n,\partial C_n)}{2}$ and denote by $B_n$ the connected component of $q_n$ in $B_{\mathcal M}(q_n,r_n)\cap E_1.$ We have $B_n\subset C_n.$ If $q\in B_n,$ then $f_n(q)\leq f_n(q_n)$ and
$$
\begin{array}{rcl}
d(q_n, \partial C_n)&\leq & d(q_n,q) + d(q, \partial C_n)\\
&&\\
&\leq & \frac{d(q_n,\partial C_n)}{2} +d(q, \partial C_n)\\
&&\\
\Rightarrow d(q_n, \partial C_n) &\leq & 2d(q,\partial C_n),

\end{array}
$$
hence we conclude that $|A(q)|\leq 2|A(q_n)|.$

Call $g$ the metric on $E$ and take $\lambda_n=|A(q_n)|.$ Consider ${\Sigma}_n$ the homothety of $B_n$ by $\lambda_n,$ that is, $\Sigma_n$ is the ball $B_n$ with the metric $g_n=\lambda_n g.$ We can use the exponential map at the point $q_n$ to lift the surface $\Sigma_n$ to the tangent plane $T_{q_n}\mathcal M\approx \rr^3,$ hence we obtain a surface $\widetilde{\Sigma}_n$ in $\rr^3$ which is a minimal surface with respect to the lifted metric $\tilde{g}_n,$ where $\tilde{g}_n$ is the metric such that the exponential map exp$_{q_n}$ is an isometry from $(\widetilde{\Sigma}_n, \tilde{g}_n)$ to $(\Sigma_n,g_n).$


We have $\widetilde{\Sigma}_n\subset B_{\rr^3}(0,\lambda_nr_n),$ $|A(0)|=1$ and $|A(q)|\leq2$ for all $q\in\widetilde{\Sigma}_n.$ 

Note that $2\lambda_nr_n=f_n(q_n)\geq f_n(p_n)\geq\delta n,$ hence $\lambda_nr_n\rightarrow+\infty$ as $n\rightarrow\infty.$

Fix $n.$ The sequence $\left\{\widetilde{\Sigma}_k\cap B_{\rr^3}(0,\lambda_nr_n)\right\}_{k\geq n}$ is a sequence of compact surfaces in $\rr^3,$ with bounded curvature, passing through the origin and the metric $g_k$ converges to the canonical metric $g_0$ in $\rr^3.$ Then a subsequence converges to a minimal surface in $(\rr^3,g_0)$ passing through the origin with the norm of the second fundamental form at the origin equal to 1. We can apply this argument for each $n$ and using the diagonal sequence argument, we obtain a complete minimal surface $\widetilde{\Sigma}$ in $\rr^3,$ with $0\in \widetilde{\Sigma}$ and $|A(0)|=1.$ In particular, $\widetilde{\Sigma}$ is not the plane. Then by Osserman's theorem \cite{O} we know $\int_{\widetilde{\Sigma}}| A|^2\geq 4\pi.$ 

We know that the integral $\int_\Sigma |A|^2$ is invariant by homothety of $\Sigma$, hence
$$
\int_{B_n} |A|^2=\int_{\Sigma_n}|A|^2=\int_{\widetilde{\Sigma}_n}|A|^2.
$$

Consider a compact $K\subset\widetilde{\Sigma}$ sufficiently large so that $\int_K|A|^2\geq2\pi.$ Fix $n$ such that $K\subset B(0,\lambda_nr_n).$ As a subsequence of $\widetilde{\Sigma}_k\cap B_{\rr^3}(0,\lambda_nr_n)$ converges to $\widetilde{\Sigma}\cap B_{\rr^3}(0,\lambda_nr_n),$ we have for $k$ sufficiently large that
$$
\int_{\widetilde{\Sigma}_k\cap B(0,\lambda_nr_n)}|A|^2\geq 2\pi-\epsilon,
$$ 
for some small $\epsilon.$ It implies $\int_{B_k}|A|^2\geq 2\pi-\epsilon,$ for $k$ sufficiently large. As $B_i\cap B_j=\emptyset,$ we conclude that $\int_{E}|A|^2=+\infty.$ But this is not possible, since 
$$
\int_{E}|A|^2=\int_{E}-2K_E+2K_{\mbox{sec}_{\mathcal M}(E)}\leq -2\int_E K_E<+\infty. 
$$

Therefore, $E$ has necessarily bounded curvature. 

Since $E$ is complete, there exist $\epsilon>0$ and a sequence of points $\{p_n\}$ in $E$ such that $p_n$ diverges in $\mathcal M_-$ and $B_E(p_k,\epsilon)\cap B_E(p_j,\epsilon)=\emptyset,$ where $B_E(p_k,\epsilon)\subset E$ is the intrinsic ball centered at $p_k$ with radius $\epsilon.$
As $E$ has bounded curvature, then there exists $\tau<\epsilon$ such that $B_E(p_k,\tau)$ is a graph with bounded geometry over a small disk $D(0,\tau)$ of radius $\tau$ in $T_{p_k}E,$ and the area of $B_E(p_k,\tau)$ is greater or equal to the area of $D(0,\tau).$ Therefore,
$$\mbox{area}(E)\geq \sum_{n\geq1}\mbox{area}\left(B_E(p_n,\tau)\right)=\infty.$$



\end{proof}



\begin{definition}
We write \textit{Helicoidal plane} to denote a minimal surface in $\hr$ which is parametrized by $X(x,y)=(x,y,ax+b)$ when we consider the halfplane model for $\hh.$ 
\end{definition}

Now we can state the main result of this paper.

\begin{theorem}
Let $X: \Sigma \hookrightarrow \mathcal M = \hr/[\psi, T(h)]$ be a properly immersed minimal surface with finite total curvature. Then 
\begin{enumerate}
\item $\Sigma$ is conformally equivalent to a compact Riemann surface $\overline{M}$ with genus $g$ minus a finite number of points, that is, $\Sigma=\overline{M}\setminus \{p_1,..., p_k\}$.
\item The total curvature satisfies
$$
\int_{\Sigma}Kd\sigma=2\pi(2-2g-k).
$$
\item If we parametrize each end by a punctured disk then either $Q$ extends to zero at the origin (in the case where the end is asymptotic to a horizontal slice) or $Q$ extends meromorphically to the puncture with a double pole and  residue zero.
In this last case, the third coordinate satisfies $h(z)=b{\rm arg}(z)+O(|z|)$ with $b \in \mathbb R$.
\item The ends contained in $\mathcal M_-$ are necessarily asymptotic to a vertical plane $\gamma\times\mathbb S^1$ and the ends contained in $\mathcal M_+$ are asymptotic to either
\begin{itemize}
\item a horizontal slice $\hh/[\psi]\times\{c\},$ or
\item a vertical plane $\gamma\times\mathbb S^1$, or 
\item the quotient of a Helicoidal plane.
 \end{itemize}
 \end{enumerate}
\end{theorem}

\begin{proof}
The proof of this theorem uses arguments of harmonic diffeomorphisms theory as can be found in the work of Han, Tam, Treibergs and Wan \cite{ZCH, HTTW, W} and Minsky \cite{M}.

From a result by Huber \cite{Hu}, we deduce that $\Sigma$ is conformally a compact Riemann surface $\overline{M}$ minus a finite number of points $\{p_1,...,p_k\},$ and the ends are parabolic. 

We consider $\overline{M}^*=\overline{M}-\cup_{i}B(p_i,r_i),$ the surface minus a finite number of disks removed around the punctures $p_i.$ As the ends are parabolic, each punctured disk $B^*(p_i,r_i)$ can be parametrized conformally by the exterior of a disk in $\mathbb C,$ say $U=\{z\in \mathbb C; |z|\geq R_0\}.$

Using the Gauss-Bonnet theorem for $\overline{M}^*,$ we get
\begin{equation}
\int_{\overline{M}^*}Kd\sigma+\sum_{i=1}^k\int_{\partial B(p_i,r_i)}k_g\ ds=2\pi(2-2g-k).
\label{eqGB}
\end{equation}

Therefore, in order to prove the second item of the theorem is enough to show that for each $i,$ we have $$\int_{\partial B(p_i,r_i)}k_g\ ds=\int_{B(p_i,r_i)}Kd\sigma.$$ In other words, we have to understand the geometry of the ends. Let us analyse each end.

Fix $i,$ denote $E=B^*(p_i,r_i)$ and let $X=(F,h):U=\{|z|\geq R_0\}\rightarrow \mathbb H^2/[\psi]\times \mathbb S^1$ be a conformal parametrization of the end $E.$ In this parameter we express the metric as $ds^2=\lambda^2|dz|^2$ with $\lambda^2=4(\mbox{cosh}^2\omega)|\phi|,$ where $\phi (dz)^2=Q$ is the holomorphic quadratic differential on the end.  

If $Q\equiv 0$ then $\phi\equiv 0$ and $h\equiv$ constant, what yields that the end $E$ of $\Sigma$ is contained in some slice $\hh/[\psi]\times\{c_0\}.$ Then, in fact, the minimal surface $\Sigma$ is the slice $\hh/[\psi]\times\{c_0\}.$ Note that by our hypothesis on $\Sigma$ this case is possible only when the horizontal slices of $\mathcal M$ have finite area. Therefore, we can assume $Q\not\equiv 0.$ 


Following the ideas of \cite{HTTW} and section 3 of \cite{HR}, we can show that finite total curvature and non-zero Hopf differential $Q$ implies that $Q$ has a finite number of isolated zeroes on the surface $\Sigma.$ Moreover, for $R_0>0$ large enough we can show that there is a constant $\alpha$ such that $(\cosh^2\omega)|\phi|\leq |z|^\alpha |\phi|$ and then, as the metric $ds^2$ is complete, we use a result by Osserman \cite{O} to conclude that $Q$ extends meromorphically to the puncture $z=\infty.$ Hence we can suppose that $\phi$ has the following form:
$$
\phi(z)=\left(\sum_{j\geq 1}\frac{a_{-j}}{z^j}+P(z)\right)^2,
$$
for $|z|>R_0$, where $P$ is a polynomial function.

Since $\phi$ has a finite number of zeroes on $U,$ we can suppose without loss of generality that $\phi$ has no zeroes on $U,$ and then the minimal surface $E$ is transverse to the horizontal sections $\hh/[\psi]\times\{c\}.$ 


As in a conformal parameter $z,$ we express the metric as $ds^2=\lambda^2|dz|^2,$ where $\lambda^2=4(\mbox{cosh}^2\omega)|\phi|,$ then on $U$
\begin{equation}
-K_{\Sigma}\lambda^2=4(\sinh^2\omega)|\phi|+\displaystyle\frac{|\nabla\omega|^2}{\cosh^2\omega}\geq 0.
\label{eqK}
\end{equation}

Hence,
$$
\begin{array}{rcl}
-\displaystyle\int_UKdA&=&\displaystyle\int_U4(\sinh^2\omega)|\phi||dz|^2+\int_U\frac{|\nabla\omega|^2}{\mbox{cosh}^2\omega}|dz|^2\\
&&\\
&=& \displaystyle\int_U4(\mbox{cosh}^2\omega)|\phi||dz|^2-\int_U4|\phi||dz|^2+\int_U\frac{|\nabla\omega|^2}{4(\mbox{cosh}^4\omega)|\phi|}dA\\
&&\\
&=& \mbox{area}(E)-4\displaystyle\int_U|\phi||dz|^2+\int_U\frac{|\nabla\omega|^2}{4(\mbox{cosh}^4\omega)|\phi|}dA,
\end{array}
$$
where the last term in the right hand side is finite by $(\ref{eqK}),$ once we have finite total curvature.

By the above equality, we conclude that area$(E)$ is finite if, and only if, $\phi=\left(\sum_{j\geq 2}\frac{a_{-j}}{z^j}\right)^2.$ Equivalently, area$(E)$ is infinite if, and only if, $\phi=\left(\sum_{j\geq1}\frac{a_{-j}}{z^j}+P(z)\right)^2,$ with $P\not\equiv 0$ or $a_{-1}\neq0.$ 

\vspace{0.2cm}

\textbf{Claim 1:} If the area of the end is infinite, then the function $\omega$ goes to zero uniformly at infinity.
\begin{proof} 
To prove this we use estimates on positive solutions of $\sinh$-Gordon equations by Han \cite{ZCH}, Minsky \cite{M} and Wan \cite{W} to our context.

Given $V$ any simply connected domain of $U=\{|z|\geq R_0\},$ we have the conformal coordinate $w=\int\sqrt{{\phi}}dz=u+iv$ with the flat metric $|dw|^2=|{\phi}||dz|^2$ on $V$.
In the case where $P\not\equiv0,$ the disk $D(w(z),|z|/2)$ contains a ball of radius at least $c|z|$ in the
metric $|dw|^2$ where $c$ does not depend on $z$.

In the case where $a_{-1}\neq 0$, we consider the conformal universal covering $\tilde U$ of the annulus $U$ given by the conformal change of coordinate $w=\ln (z) + f(z)$ where $f(z)$ extends holomorphically by zero at
the puncture. Any point $z$ in $U$ lifts to the center $w(z)$ of a ball $D(w(z), \ln( |z|/2)) \subset \tilde U$  for $|z| >2R_0$ large enough. 

The function $\omega$ lifts to the function $\tilde \omega \circ w (z):=\omega (z)$ on the $w$-plane which satisfies the equation
$$\Delta_{|\phi|} \tilde \omega= 2 \sinh 2  \tilde \omega $$
where $\Delta_{|\phi|}$ is the Laplacian in the flat metric $|dw|^2$. On the disc $D (w(z), 1)$ we consider 
the hyperbolic metric given by



$$
d\sigma^2=\mu^2|dw|^2=\frac{4}{(1-|w-w(z) |^2)^2}|dw|^2.
$$

Then $\mu$ takes infinite values on $\partial D(w(z),1)$ and since the curvature of the metric $d\sigma^2$ is $K=-1,$ the function $\omega_2=\ln\mu$ satisfies the equation
$$
\Delta_{|\phi|}\omega_2=e^{2\omega_2}\geq e^{2\omega_2}-e^{-2\omega_2}=2\sinh\omega_2,
$$
Then the function $\eta(w) =\tilde\omega (w)-\omega_2 (w)$ satisfies
$$
\Delta_{|\phi|}\eta=e^{2\tilde\omega}-e^{-2\tilde\omega}-e^{2\omega_2}=e^{2\omega_2}\left(e^{2\eta}-e^{-4\omega_2}e^{-2\eta}-1\right),
$$
which can be written in the metric $d\tilde\sigma^2=e^{2\omega_2}|dw|^2$ as
$$
\Delta_{\tilde\sigma}\eta=e^{2\eta}-e^{-4\omega_2}e^{-2\eta}-1.
$$

Since $\omega_2$ goes to $+\infty$ on the boundary of the disk $D_{|\phi|}(w(z),1),$ the function $\eta$ is bounded above and attains its maximum at an interior point $q_0.$ At this point $\eta_0=\eta(q_0)$ we have
$$
e^{2\eta_0}-e^{-4\omega_2}e^{-2\eta_0}-1\leq0.
$$
which implies
$$
e^{2\eta_0}\leq \frac{1+\sqrt{1+4a^2}}{2},
$$
where $a=e^{-2\omega_2(q_0)}\leq \ \mbox{sup}\frac{1}{\mu^2}\leq \frac{1}{4}.$ Thus at any point of the disk $D_{|\phi|}(z,1),$ $\tilde\omega$ satisfies
$$
\tilde{\omega}\leq\omega_2+\frac{1}{2}\ln (\frac{2+\sqrt{5}}{4}).
$$
We observe that the same estimate above holds for $-\tilde\omega.$ Then at the point $z,$ we have
$$
|{\omega}(z)|=|\tilde \omega (w(z))| \leq \ln4+\frac{1}{2}\ln(\frac{2+\sqrt{5}}{4}):=K_0
$$
uniformly on $R\geq R_0.$ Using this estimate we can apply a maximum principle as in Minsky \cite{M}. We know that for $|z|$ large, we can find a disk $D_{|\phi|}(w(z),r)$ with $r$ large too. Now, consider the function
$$
F(u,v)=\frac{K_0}{\cosh r}\cosh \sqrt{2}u\cosh\sqrt{2}v.
$$

Then $F\geq K_0\geq \omega$ on $\partial D_{|\phi|}(w(z),r)$ and at $q_0$ we have $\Delta_{|\phi|}F=4F.$ Suppose the minimum of $F-\tilde\omega$ is a point $q_0$ where $\tilde\omega(q_0)\geq F(q_0).$ Then $0\leq \tilde\omega(q_0)\leq \sinh\tilde\omega(q_0)$ and
$$
\Delta_{|\phi|}(F-\tilde\omega)=4F-2\sinh2\tilde\omega\leq 4(F(q_0)-\tilde\omega(q_0))\leq0.
$$ 
Therefore we have necessarily $\tilde\omega\leq F$ on the disk. Considering the same argument to $F+\tilde\omega$ we can conclude $|\tilde\omega|\leq F.$ Hence
\begin{equation}
|\tilde\omega(w(z))|\leq \frac{K_0}{\cosh r}
\label{omegaK}
\end{equation}
and then $|\tilde\omega|\rightarrow0$ uniformly at the puncture, consequently $|\omega|\rightarrow0$ uniformly at infinity.
\end{proof}

\vspace{0.2cm}

\textbf{Claim 2:} If $P\not\equiv0$ then the end $E$ is not proper in $\mathcal M.$
\begin{proof}
Suppose $P\not\equiv0.$ Up to a change of variable, we can assume that the coefficient of the leading term of $P$ is one. Then, for suitable complex number $a_0, . . . , a_{k-1},$ we have
$$
P(z) = z^k + a_{k-1}z^{k-1} + · · · + a_0 \ \mbox{and} \ \sqrt{\phi}=z^{k}(1 + o(1)).$$

Let us define the function $$w(z)=\int\sqrt{\phi(z)}dz=\int\left(\sum_{j\geq1}\frac{a_{-j}}{z^j}+a_0+...+z^k\right).$$ 

If $a_{-1}=a+ib$ and we denote by $\theta\in\rr$ a determination of the argument of $z\in U,$ then locally 
\begin{equation}
\mbox{Im}(w)(z)=b\mbox{log}|z|+a\theta+\frac{|z|^{k+1}}{k+1}(\mbox{sin}(k+1)\theta+o(1))
\label{imW}
\end{equation}
and
\begin{equation}
\mbox{Re}(w)(z)=a\mbox{log}|z|-b\theta+\frac{|z|^{k+1}}{k+1}(\mbox{cos}(k+1)\theta+o(1)).
\label{reW}
\end{equation}

If $C_0 > \mbox{max}\{|\mbox{Im}(w)(z)|; |z|=R_0\},$ then the set $U\cap\{\mbox{Im}(w)(z) = C_0\}$ is composed of $k + 1$ proper and complete curves without boundary $L_0 , . . . , L_k$ (see Figure \ref{Hk1}).

\begin{figure}[h]
 \centering
\includegraphics[height=5cm]{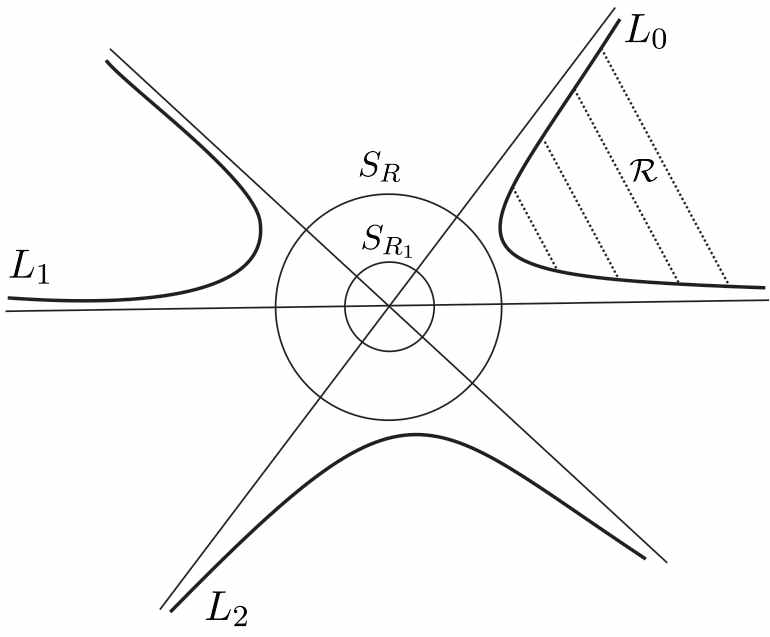}
\caption{$L_j$ for $k=2$.}
\label{Hk1}
\end{figure}

Take $\mathcal R$ a simply connected component of $U\cap\{\mbox{Im}(w)(z) \geq C_0\}.$ The holomorphic map $w(z)$ gives conformal parameters $w=u+iv, v\geq C_0,$ to $X(\mathcal R)\subset E.$  

Then $\widetilde X(w)=(\widetilde F(w),v)$ is a conformal immersion of $\mathcal R$ in $\hr$ and we have
$$
|\widetilde F_u|^2_{\sigma}=|\widetilde F_v|^2_\sigma+1 \ \mbox{and} \ \left\langle\widetilde F_u,\widetilde F_v\right\rangle_\sigma=0.
$$
Hence the holomorphic quadratic Hopf differential is
$$
Q_{\widetilde F}=\phi(w)(dw)^2=\frac{1}{4}\left(|\widetilde F_u|^2_{\sigma}-|\widetilde F_v|^2_\sigma+2i\left\langle \widetilde F_u,\widetilde F_v\right\rangle_\sigma\right)=\frac{1}{4}(dw)^2
$$
and the induced metric on these parameters is given by $ds^2=\mbox{cosh}^2\widetilde\omega|dw|^2.$

Consider the curve $\gamma(v)=\widetilde X(u_0+iv)=(\widetilde F(u_0,v),v).$ We have

$$d_{\hh}(\widetilde F(u_0,C_0),\widetilde F(u_0,v))\leq \int_{C_0}^v|\widetilde{F}_v|dv=\int_{C_0}^v|\sinh\tilde{\omega}|dv<\infty,$$
once we know $|\tilde{\omega}|\rightarrow0$ at infinity by Claim 1.

Thus, when we pass the curve $\gamma$ to the quotient by the third coordinate, we obtain a curve in $E$ which is not properly immersed in the quotient space $\mathcal M.$ Therefore, the claim is proved and we have $P\equiv0$ necessarily.
\end{proof}

\vspace{0.2cm}


Suppose $E\subset \mathcal M_+.$ We have $E=X(U)$ homeomorphic to $\mathbb S^1\times \rr.$ Up to translation (along a geodesic not contained in $\tor(0)$), we can suppose that $E$ is transverse to $\tor(0).$ Then $E\cap\tor(0)$ is $k$ jordan curves $d_1,...,d_j,\alpha_1,...,\alpha_l,j+l=k,$ where each $d_i$ is homotopically zero in $E$ and each $\alpha_i$ generates the fundamental group of $E,\pi_1(E).$

We will prove that $l=1$ necessarily and the subannulus bounded by $\alpha_1$ is contained in $\cup_{s\geq0}\tor(s)$.

Assume $l\neq1.$ Then there exist $\alpha_1,\alpha_2\subset\tor(0)$ generators of $\pi_1(E).$ As $E\cong \mathbb S^1\times\rr,$ there exists $F\subset E$ such that $F\cong \mathbb S^1\times[0,1]$ and $\partial F=\alpha_1\cup \alpha_2.$ So $F$ is compact and its boundary is on $\tor(0).$ 
By the maximum principle, $F\cap\left(\cup_{s<0}\tor(0)\right)=\emptyset.$ Hence $F\subset\cup_{s\geq 0}\tor(s)$ and then, since $E\subset \mathcal M_+,$ there exist a third jordan curve $\alpha_3$ that generates $\pi_1(E)$ and another cylinder $G$ such that $G\cap\left(\cup_{s<0}\tor(0)\right)\neq\emptyset$ and $\partial G$ is either $\alpha_1\cup \alpha_3$ or $\alpha_2\cup \alpha_3,$ but we have just seen that such $G$ can not exist. Therefore $l=1,$ that is, $E\cap\tor(0)=\alpha\cup d_1\cup ... \cup d_j,$ where $\alpha$ generates $\pi_1(E),$ the subannulus bounded by $\alpha$ is contained in $\cup_{s\geq0}\tor(s),$ and each $d_i\subset E$ bounds a disk on $E$ contained in $\cup_{s\geq0}\tor(s).$ 

\begin{remark}
The same holds true for $E\subset \mathcal M_-,$ that is, if $E\subset \mathcal M_-$ and $E$ is transversal to $\tor(s)$ then $E\cap\tor(s)$ is $l_s+1$ curves $\alpha,d_1,...,d_{l_s},$ where $d_i$ is homotopically zero in $E$ and $\alpha$ generates $\pi_1(E).$
\label{remark1}
\end{remark}

Take a point $p$ in the horocycle $c(0)\subset \mathbb H^2$ and consider $e_1=c(0)/[\psi],$ $e_2=p\times\rr/[T(h)].$ The curves $e_1,e_2$ are generators of $\pi_1(\tor(0)).$

As $E\subset \mathcal M_+$ and $\pi_1(\mathcal M_+)=\pi_1(\tor(0)),$ we can consider the inclusion map $i_*:\pi_1(E)\rightarrow \pi_1(\tor(0))$ and $i_*([\alpha])=n[e_1]+m[e_2],$ where $m,n$ are integers.



\vspace{0.3cm}

\textbf{Case 1.1:} $n=m=0.$ This case is impossible.

In fact, $n=m=0$ implies that $E$ lifts to an annulus in $\hr$ and we already know by Lemma \ref{annulus} that this is not possible.



\vspace{0.3cm}

\textbf{Case 1.2:} $n \neq0,m=0.$

We can assume, without loss of generality, that $\partial E\subset \tor(0)$. Call $\widetilde E$ a connected component of $\pi^{-1}(E\cap \mathcal M_+)$ such that $\pi(\widetilde {E})=E.$ We have that $\widetilde E$ is a proper minimal surface and its boundary $\partial\widetilde{E}=\pi^{-1}(\partial E)$ is a curve in $d(0)$ invariant by $\psi^n.$ Moreover, the horizontal projection of $\widetilde{E}$ on $\cup_{s\geq 0}c(s)\subset\hh$ is surjective.



By the Trapping Theorem in \cite{CHR2}, $\widetilde E$ is contained in a horizontal slab. Hence $\widetilde {h}|_{\widetilde E}$ is a bounded harmonic function, and then $h|_{E}$ is a bounded harmonic function defined on a punctured disk. Therefore $h$ has a limit at infinity, and then we can say that $Q$ extends to a constant at the origin, say zero. In particular, $\widetilde h$ has a limit at infinity.

The end of $\widetilde E$ is contained in a slab of width $2\epsilon>0$ and by a result of Collin, Hauswirth and Rosenberg \cite{CHR}, $\widetilde E$ is a graph outside a compact domain of $\hr.$ This implies that $\widetilde E$ has bounded curvature. Then there exists $\delta>0$ such that for any $p\in E,$ $B_E(p,\delta)$ is a minimal graph with bounded geometry over the disk $D(0,\delta)\subset T_pE.$

Now fix $s$ and consider a divergent sequence $\{p_n\}$ in $E.$ Applying hyperbolic translations to $\{p_n\}$(horizontal translations along a geodesic of $\hh$ that sends $p_n$ to a point in $\tor(s)$), we get a sequence of points in $\tor(s)$ which we still call $\{p_n\}.$ As $\tor(s)$ is compact, the sequence $\{p_n\}$ converges to a point $p\in\tor(s)$ and the sequence of graphs $B_E(p_n,\delta)$ converges to a minimal graph $B_E(p,\delta)$  with bounded geometry over $D(0,\delta)\subset T_pE.$

As $h$ has a limit at infinity, this limit disk $B_E(p,\delta)$ is contained in a horizontal slice. Then we conclude $n_3\rightarrow 1$ and $|\nabla h|\rightarrow 0$ uniformly at infinity, what yields a $C^1$-convergence of $E$ to a horizontal slice. Now using elliptic regularity we get $E$ converges in the $C^2$-topology to a horizontal slice. In particular, the geodesic curvature of $\alpha_s$ goes to $1$ and its length goes to zero, where $\alpha_s$ is the curve in $E\cap\tor(s)$ that generates $\pi_1(E).$ 


Denote by $E_s$ the part of the end $E$ bounded by $\partial E$ and $\alpha_s.$ Applying the Gauss-Bonnet theorem for $E_s,$ we obtain
$$
\int_{E_s}K + \int_{\alpha_s}k_g -\int_{\partial E}k_g=0.
$$
By our analysis in the previous paragraph, we have $\int_{\alpha_s}k_g\rightarrow0,$ when $s\rightarrow\infty.$ Then when we let $s$ go to infinity, we get
$$
\int_E K=\int_{\partial E}k_g,
$$
as we wanted to prove.

\vspace{0.2cm}

\textbf{Claim 3:} If $m\neq0$ then the area of the end is infinite.
\begin{proof} 
In fact, consider $g:\Sigma\rightarrow \rr$ the extrinsic distance function to $\tor(0),$ that is, $g=d_{\mathcal M}(\ .\ , \tor(0)).$ Hence $|\nabla^{\mathcal M} g|=1$ and $g^{-1}(s)=\Sigma\cap\tor(s).$ We know for almost every $s,$ $\Sigma\cap\tor(s)=\alpha_s\cup d_1\cup ...\cup d_l,$ where $\alpha_s$ generates $\pi_1(E)$ and $d_i$ is homotopic to zero in $E$. Then, by the coarea formula,
$$
\begin{array}{rclrl}
\displaystyle\int_{\{g\leq s\}}1dA&=&\displaystyle\int_{-\infty}^s\left(\int_{\{g=\tau\}}\frac{ds_{\tau}}{|\nabla^{\Sigma} g|}\right)d\tau &\geq&\displaystyle \int_{0}^s |\alpha_\tau|d\tau\\
&&\\
&\geq& \displaystyle\int_0^s |e_2|d\tau \ =\ s|e_2|,
\end{array}
$$
where the last inequality follows from the fact we are supposing that $i_{*}[\alpha_s]$ has a component $[e_2],$ and in the last equality we use that the curve $e_2$ has constant length. Hence when we let $s$ go to infinity, we conclude the area of $E$ is infinite.
\end{proof}

So if $E\subset\mathcal M_+$ and $m\neq0,$ then the area of $E$ is infinite. Also, we know by Lemma \ref{prop-area} that all the ends contained in $\mathcal M_-$ have infinite area. Thus we will analyse all these cases together using the commom fact of infinite area.

Suppose we have an end $E$ with infinite area. We can assume without loss of generality that $\partial E\subset \tor(0).$ We know that $\phi=\left(\sum_{j\geq1}\frac{a_{-j}}{z^j}\right)^2$ with $a_{-1}\neq0$ for $|z|\geq R_0$, and $|\omega|\rightarrow0$ uniformly at infinity by Claim 1. In particular, we know that the tangent planes to the end become vertical at infinity. 

Let $X: D^*(0,1)\subset\mathbb C\rightarrow \mathcal M $ be a conformal parametrization of the end from a punctured disk (we suppose, without loss of generality, that the punctured disk is the unit punctured disk). Now consider the covering of $D^*(0,1)$ by the halfplane $HP:=\{w=u+iv, u<0\}$ through the holomorphic exponential map $e^w:HP\rightarrow D^*(0,1).$ Hence, we can take $\hat{X}=X\circ e^w:HP\rightarrow \mathcal M$ a conformal parametrization of the end from a halfplane.

We denote by $h,\hat{h}$ the third coordinates of $X$ and $\hat{X},$ respectively. We already know $h(z)=a\ln|z|+b\mbox{arg}(z)+p(z)$ for $z\in D^*(0,1),$ where either $a$ or $b$ is not zero, and $p$ is a polynomial function. Hence $|p(z)|\rightarrow0$ when $|z|\rightarrow0$ and $\hat{h}(w)=au+bv+\hat{p}(w),$ where $u=\mbox{Re}\ (w),v=\mbox{Im}\ (w)$ and $\hat{p}(w)=p(e^w).$

As the halfplane is simply connected, consider $\widetilde {X}:HP\rightarrow\hr$ the lift of $\hat {X}$ into $\hr.$ We have $\widetilde {X}=(\widetilde F,\widetilde h),$ where $\widetilde h(w)=au+bv+\widetilde p(w),$ with $|\widetilde p(w)|\rightarrow0$ when $|w|\rightarrow\infty.$ Up to a conformal change of parameter, we can suppose that $\widetilde{h}(w)=au+bv.$

Observe $\partial \widetilde E=\widetilde X(\{u=0\})$ and the curve $\{\widetilde h=c\}$ is the straight line $\{au+bv=c\}.$ We have three cases to analyse.

\vspace{0.3cm}

\textbf{Case 2.1:} $a=0,b\neq0,$ that is, the third coordinate satisfies $h(z)=b\mbox{arg}(z)+O(|z|).$

Without loss of generality we can suppose $b=1.$ Hence in this case, $\widetilde h(w)=v$ and $\partial\widetilde{E}=\widetilde{X}(\{u=0\}).$

We have $\widetilde X(w)=(\widetilde F(w),v)$ a conformal immersion of $\widetilde E,$ and
$$
|\widetilde F_u|^2_{\sigma}=|\widetilde F_v|^2_\sigma+1 \ \mbox{and} \ \left\langle\widetilde F_u,\widetilde F_v\right\rangle_\sigma=0.
$$
Hence the holomorphic quadratic Hopf differential is
$$
\widetilde Q_{\widetilde F}=\tilde\phi(w)(dw)^2=\frac{1}{4}\left(|\widetilde F_u|^2_{\sigma}-|\widetilde F_v|^2_\sigma+2i\left\langle \widetilde F_u,\widetilde F_v\right\rangle_\sigma\right)=\frac{1}{4}(dw)^2
$$
and the induced metric on these parameters is given by $ds^2=\mbox{cosh}^2\widetilde\omega|dw|^2.$

Moreover, by $(\ref{omegaK})$ there exists a constant $K_0>0$ such that
\begin{equation}
|\widetilde\omega(w)|\leq \frac{K_0}{\mbox{cosh}r},
\label{w-estimate}
\end{equation}
for $r=\sqrt{u^2+v^2}$ sufficiently large.

Using Schauder's estimates and (\ref{w-estimate}), we obtain 
$$
|\widetilde\omega|_{2,\alpha}\leq C\left(|\sinh\widetilde\omega|_{0,\alpha}+|\widetilde\omega|_0\right)\leq Ce^{-r}.
$$
Then
\begin{equation}
|\nabla\widetilde\omega|\leq Ce^{-r}.
\label{grad-w}
\end{equation}

Now consider the curve $\gamma_c=\widetilde E\cap\hh\times\{v=c\},$ that is, $\gamma_c(u)=(\widetilde F(u,c),c).$ Let $(V,\sigma(\eta)|d\eta|^2)$ be a local parametrization of $\hh$ and define the local function $\varphi$ as the argument of  $\widetilde F_u,$ hence
$$
\widetilde F_u=\frac{1}{\sqrt{\sigma}}\cosh \widetilde\omega e^{i\varphi} \ \mbox{and} \ \widetilde F_v=\frac{i}{\sqrt{\sigma}}\sinh\widetilde \omega e^{i\varphi}.
$$

If we denote by $k_g$ the geosedic curvature of $\gamma_c$ in $(V,\sigma(\eta)|d\eta|^2)$ and by $k_e$ the Euclidean geodesic curvature of $\gamma_c$ in $(V,|d\eta|^2),$ we have
$$
k_g=\frac{k_e}{\sqrt{\sigma}}-\frac{\left\langle \nabla \sqrt{\sigma}, n\right\rangle}{\sigma},
$$ 
where $n=(-\sin\varphi,\cos\varphi)$ is the Euclidean normal vector to $\gamma_c.$ If $t$ denotes the arclength of $\gamma_c,$ we have
$$
k_e=\varphi_t=\frac{\varphi_u\sqrt{\sigma}}{\cosh\widetilde\omega}
$$
and
$$
\frac{\left\langle \nabla \sqrt{\sigma}, n\right\rangle}{\sigma}= \frac{\left\langle \nabla\mbox{log} \sqrt{\sigma}, n\right\rangle}{\sqrt{\sigma}}=\frac{1}{2\sqrt{\sigma}}\left(\cos\varphi(\log\sigma)_{\eta_2}-\sin\varphi(\log\sigma)_{\eta_1} \right).
$$
Then,
\begin{equation}
k_g=\frac{\varphi_u}{\cosh\widetilde\omega}-\frac{1}{2\sqrt{\sigma}}\left(\cos\varphi(\log\sigma)_{\eta_2}-\sin\varphi(\log\sigma)_{\eta_1} \right).
\label{kg}
\end{equation}

In the complex coordinate $w,$ we have
\begin{equation}
\widetilde F_w=\frac{e^{\widetilde\omega+i\varphi}}{2\sqrt{\sigma}} \ \mbox{and} \ \widetilde F_{\bar{w}}=\frac{e^{-\widetilde\omega+i\varphi}}{2\sqrt{\sigma}}.
\label{eqFw}
\end{equation}
Moreover, the harmonic map equation in the complex coordinate $\eta=\eta_1+i\eta_2$ of $\hh$ (see \cite{SY}, page 8) is
\begin{equation}
\widetilde F_{w\bar{w}}+(\log\sigma)_{\eta}\widetilde F_w\widetilde F_{\bar w}=0.
\label{eq11}
\end{equation}
Then using $(\ref{eqFw})$ and $(\ref{eq11})$ we obtain
\begin{equation}
\begin{array}{rcl}
(-\widetilde\omega+i\varphi)_w&=&-\sqrt{\sigma}\left(\frac{1}{\sqrt\sigma}\right)_w-(\log \sigma)_{\eta}\widetilde F_w\\
&&\\
&=&\frac{1}{2}(\log \sigma)_w-(\log \sigma)_{\eta}\widetilde F_w\\
&&\\
&=&\frac{1}{2}\left((\log\sigma)_{\eta}\widetilde F_w+(\log\sigma)_{\bar {\eta}}\bar{\widetilde F}_w\right)-(\log \sigma)_{\eta}\widetilde F_w\\
&&\\
&=&\frac{1}{2}(\log\sigma)_{\bar\eta}\bar{\widetilde F}_w-\frac{1}{2}(\log\sigma)_{\eta}\widetilde F_w,
\end{array}
\label{eq12}
\end{equation}
where $2(\log\sigma)_{\eta}=(\log\sigma)_{\eta_1}-i(\log\sigma)_{\eta_2}$ and $\bar{\widetilde F_w}=\frac{1}{2\sqrt\sigma}e^{-\widetilde\omega-i\varphi}.$

Taking the imaginary part of $(\ref{eq12}),$ we get
\begin{equation}
\varphi_u+\widetilde\omega_v=\frac{\cosh\widetilde\omega}{2\sqrt\sigma}\left(\cos\varphi(\log\sigma)_{\eta_2}-\sin\varphi(\log\sigma)_{\eta_1}\right).
\label{eq13}
\end{equation}

By $(\ref{kg})$ and $(\ref{eq13}),$ we deduce
\begin{equation}
k_g=-\frac{\widetilde\omega_v}{\cosh\widetilde\omega}.
\label{eqKg}
\end{equation}

Therefore, by $(\ref{w-estimate})$ and $(\ref{grad-w}),$ when $c\rightarrow+\infty,$ $k_g(\gamma_c)(u)\rightarrow0$ and also when we fix $c$ and let $u$ go to infinity the geodesic curvature of the curve $\gamma_c$ goes to zero. In particular, for $\widetilde h$ sufficiently large, the asymptotic boundary of $\gamma_c$ consists in only one point (see \cite{HNST}, Proposition 4.1).

We will prove that the family of curves $\gamma_c$ has the same boundary point at infinity independently on the value $c.$ Fix $u_0$ and consider $\alpha_{u_0}$ the projection onto $\hh$ of the curve $\widetilde X(u_0,v)=(\widetilde F(u_0,v),v),$ that is, $\alpha_{u_0}(v)=\widetilde F(u_0,v)\in\hh.$ We have $\alpha_{u_0}'(v)=\widetilde F_v$ and $|\alpha_{u_0}'(v)|_{\sigma}=|\sinh\widetilde\omega|.$ Then
$$
d(\alpha_{u_0}(v_1),\alpha_{u_0}(v_2))\leq l(\alpha_{u_0}{|_{[v_1,v_2]}})=\int_{v_1}^{v_2}|\sinh\widetilde\omega|dv\leq \int_{v_1}^{v_2} \sinh {\rm e}^{-r} dv,
$$
where $r=\sqrt{u_0^2+v^2}.$ Thus, for any $v_1,v_2,$ we have $d(\alpha_{u_0}(v_1),\alpha_{u_0}(v_2))\rightarrow0$ when $u_0\rightarrow-\infty.$


Therefore, the asymptotic boundary of all horizontal curves $\gamma_c$ in $\widetilde E$ coincide, and we can write $\partial_{\infty}\widetilde {E}=p_0\times\rr$. 

Observe that as $\widetilde h|_{\partial\widetilde E}$ is unbounded, then we have two possibilities for $\partial \widetilde{E},$ either $\partial \widetilde{E}$ is invariant by a vertical translation or is invariant by a screw motion $\psi^n\circ T(h)^m, n,m\neq 0.$

\textbf{Subcase 2.1.1:} $\partial \widetilde{E}$ invariant by vertical translation and $E\subset \mathcal M_+.$

In this case, by the Trapping Theorem in \cite{CHR2}, $\widetilde E$ is contained in a slab between two vertical planes that limit to the same vertical line at infinity, $p_0\times\rr$. Moreover, since $|\widetilde\omega|\rightarrow 0$, then we get bounded curvature by (\ref{eqKK}). The same holds true for $E$ in $\mathcal M_+.$

Thus, using the same argument as in Case 1.2, we can show that in fact $E$ converges in the $C^2$-topology to a vertical plane. Therefore, the geodesic curvature of $\alpha_s$ goes to zero and its length stays bounded, where $\alpha_s$ is the curve in $E\cap\tor(s)$ that generates $\pi_1(E).$

Applying the Gauss-Bonnet theorem for $E_s,$ the part of the end $E$ bounded by $\partial E$ and $\alpha_s,$ we obtain
$$
\int_{E_s}K + \int_{\alpha_s}k_g -\int_{\partial E}k_g=0.
$$
By our analysis in the previous paragraph, we have $\int_{\alpha_s}k_g\rightarrow0,$ when $s\rightarrow\infty.$ Then, when we let $s$ go to infinity, we get
$$
\int_E K=\int_{\partial E}k_g,
$$
as we wanted to prove.

\textbf{Subcase 2.1.2:} $\partial \widetilde{E}$ invariant by vertical translation and $E\subset \mathcal M_-.$

As $\partial \widetilde{E}$ invariant by vertical translation, then we can find a horizontal geodesic $\gamma$ in $\hh$ such that $\gamma$ limits to $p_0$ at infinity and $\gamma\times\rr$ does not intersect $\partial\widetilde E.$ Call $q_0$ the other endpoint of $\gamma.$ Take $q\in\partial_{\infty}\hh$ contained in the halfspace determined by $\gamma\times\rr$ that does not contain $\partial \widetilde E$. As the asymptotic boundary of $\widetilde E$ is just $p_0\times\rr,$ then $\overline{qq_0}\times\rr$ does not intersect $\widetilde E$ for $q$ sufficiently close to $q_0.$ Also note that for any $q,$ $\overline{qq_0}\times\rr$ can not be tangent at infinity to $\widetilde E,$ because $E$ is proper in $\mathcal M.$ Thus, if we start with $q$ close to $q_0$ and let $q$ go to $p_0,$ we conclude that in fact $\gamma\times\rr$ does not intersect $\widetilde E,$ by the maximum principle. Now if we consider another point $\bar{q_0}\in\partial_{\infty}\hh$ contained in the same halfspace determined by $\gamma\times\rr$ as $\partial \widetilde E$ and such that $\bar{\gamma}\times\rr=\overline{\bar{q_0}p_0}\times\rr$ does not intersect $\partial \widetilde E,$ we can prove using the same argument above that $\bar{\gamma}\times\rr$ does not intersect $\widetilde E.$ Thus we conclude that $\widetilde E$ is contained in the region between two vertical planes that limit to $p_0\times\rr$.

As $|\widetilde\omega|\rightarrow 0$, we get bounded curvature by (\ref{eqKK}). So $E\subset\mathcal M_-$ is a minimal surface with bounded curvature contained in a slab bounded by two vertical planes that limit to the same point at infinity. Hence, using the same argument as in Case 1.2, we can show that $E$ converges in the $C^2$-topology to a vertical plane. Therefore, as in Subcase 2.1.1 above, we get
$$
\int_E K=\int_{\partial E}k_g.
$$

\textbf{Subcase 2.1.3:} $\partial \widetilde{E}$ invariant by screw motion and $E\subset \mathcal M_+.$

In this case, by the Trapping Theorem in \cite{CHR2}, $\widetilde E$ is contained in a slab between two parallel Helicoidal planes and, since $|\widetilde\omega|\rightarrow 0$, we get bounded curvature by (\ref{eqKK}). Then $E$ is a minimal surface in $\mathcal M_+$ with bounded curvature contained in a slab between the quotient of two parallel Helicoidal planes.

Thus, using the same argument as in Case 1.2, we can show that in fact $E$ converges in the $C^2$-topology to the quotient of a Helicoidal plane. In particular, the geodesic curvature of $\alpha_s$ goes to zero and its length stays bounded, where $\alpha_s$ is the curve in $E\cap\tor(s)$ that generates $\pi_1(E).$

Applying the Gauss-Bonnet theorem for $E_s,$ the part of the end $E$ bounded by $\partial E$ and $\alpha_s,$ we obtain
$$
\int_{E_s}K + \int_{\alpha_s}k_g -\int_{\partial E}k_g=0.
$$
By our previous analysis, we have $\int_{\alpha_s}k_g\rightarrow0,$ when $s\rightarrow\infty.$ Then, when we let $s$ go to infinity, we get
$$
\int_E K=\int_{\partial E}k_g,
$$
as we wanted to prove.

\textbf{Subcase 2.1.4:} $\partial \widetilde{E}$ invariant by screw motion and $E\subset \mathcal M_-.$

By Remark \ref{remark1}, we know that for almost every $s\leq 0,$ $\widetilde E\cap d(s)$ contains a curve invariant by screw motion, so it is not possible to have $p_0\times\rr$ as the only asymptotic boundary. Thus this subcase is not possible.

\vspace{0.3cm}

\vspace{0.3cm}

\textbf{Case 2.2:} $a\neq0.$ We will show this is not possible.

Consider the change of coordinates by the rotation $e^{i\theta}w: HP\rightarrow \widetilde {HP},$ where $\tan\theta=\frac{a}{b}$ (notice that if $b=0,$ then $\theta=\pi/2$) and $\widetilde{HP}=e^{i\theta}(HP)\subset\{\tilde{w}=\tilde{u}+i\tilde{v}\}.$ From now on, when we write one curve in the plane $\tilde w=\tilde u+i\tilde v,$ we mean the part of this curve contained in $\widetilde{HP}.$

In this new parameter $\tilde{w},$ we have $\partial \widetilde E=\widetilde X(\{b\tilde{u}+a\tilde{v}=0\}),$ the curve $\{\widetilde h=c\}$ is the straight line $\{\tilde{v}=\frac{c}{\sqrt{a^2+b^2}}\}.$ (See Figure \ref{case2-3}).

\begin{figure}[h]
\centering
\includegraphics[height=4cm]{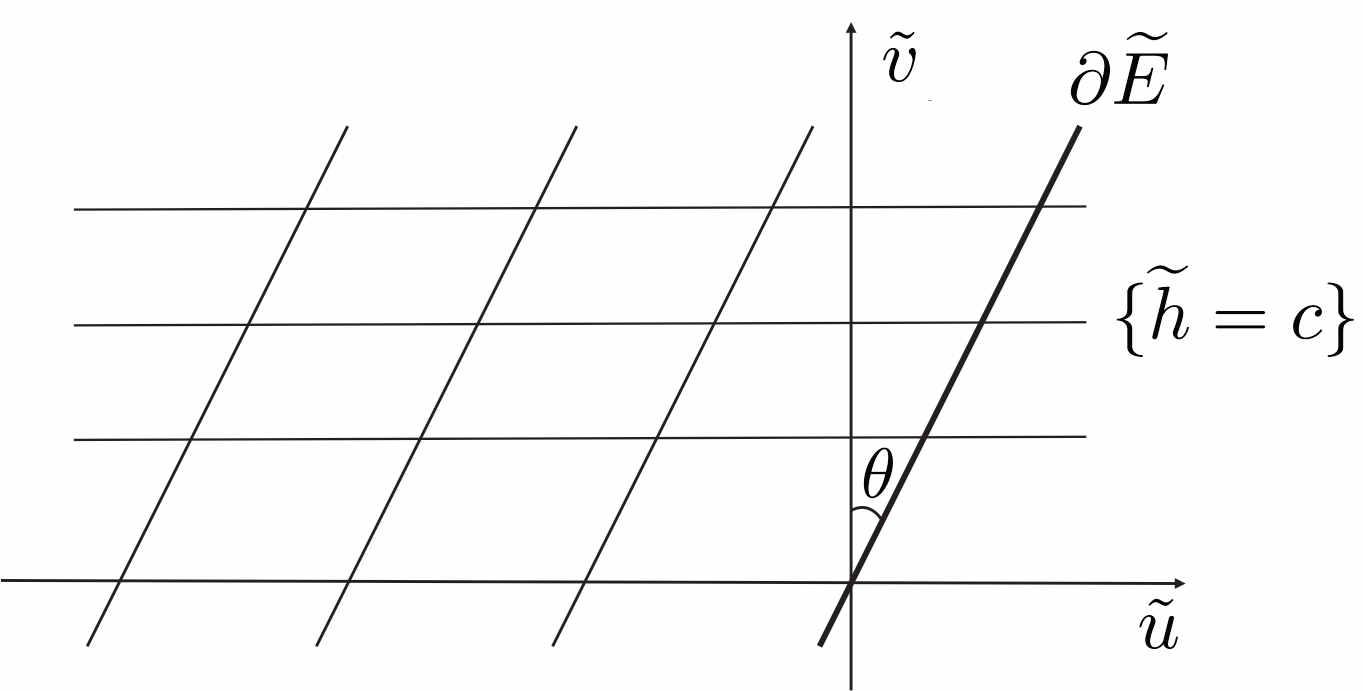}
\caption{Parameter $\tilde w=\tilde u+i\tilde v.$}
\label{case2-3}
\end{figure}

Now consider the curve $\beta(t)=(0,t), t\geq 0.$ The angle between $\widetilde{X}(\beta)$ and $\partial \widetilde E$ is $\theta\neq0$ and $\widetilde{X}(\beta)$ is a divergent curve in $\widetilde E.$ However, the curve $\widetilde{F}(\beta)=\widetilde{F}(0,t)$ satisfies
$$
l(\widetilde{F}(\beta))=\frac{1}{|a|}\int_{0}^{t}|\widetilde{F}_{\tilde{v}}|d\tilde{v}=\frac{1}{|a|}\int_{0}^{t}|\sinh\widetilde{\omega}|d\tilde{v}\leq C,
$$
for some constant $C$ not depending on $t,$ since we know by $(\ref{omegaK})$ that $|\tilde\omega|\rightarrow0$ at infinity. This implies that when we pass the curve $\widetilde{X}(\beta)$ to the quotient space $\mathcal M,$ we obtain a curve in $E$ which is not proper in $\mathcal M,$ what is impossible, once the end $E$ is proper.

\vspace{0.5cm}

Therefore, analysing the geometry of all possible cases for the ends of a proper immersed minimal surface with finite total curvature $\Sigma$ in $\mathcal M,$ we have proved the theorem.
\end{proof}
\begin{remark}
The case of a Helicoidal end contained in $\mathcal M_+$ is in fact possible, as shows the example constructed by the second author in section 4.3 in \cite{Me}. The example is a minimal surface contained in $\mathcal M$ with two vertical ends and two Helicoidal ends.
\end{remark}

\begin{flushleft}
LAURENT HAUSWIRTH, Universit\'e Paris-Est, LAMA (UMR 8050), UPEMLV, UPEC, CNRS, F-77454, Marne-la-Vall\'ee, France.

\textit{E-mail address:} laurent.hauswirth@univ-mlv.fr
\end{flushleft}

\begin{flushleft}
ANA MENEZES, Instituto de Matem\'atica Pura e Aplicada (IMPA), Estrada Dona Castorina 110, 22460-320 Rio de Janeiro, Brazil

\textit{E-mail address:} anamaria@impa.br
\end{flushleft}

\end{document}